\documentclass[11pt]{amsart}
\usepackage{amsmath,amssymb,amsfonts,amsthm,amscd,indentfirst}
\usepackage{hyperref}\usepackage{xcolor}

\numberwithin{equation}{section}

\newtheorem{prop}{Proposition}[section]

\newtheorem{lemm}[prop]{Lemma}
\newtheorem{coro}[prop]{Corollary}
\newtheorem{rem}[prop]{Remark}

\def\and{\quad{\rm and}\quad}

\def\<{\langle}
\def\>{\rangle}

\title[Curvature estimates for immersed hypersurfaces]{Curvature estimates for immersed hypersurfaces in Riemannian manifolds}
\author{Pengfei Guan
        \and
        Siyuan Lu
        }
\address{Department of Mathematics and Statistics, McGill University, 805 Sherbrooke O, Montreal, Quebec, Canada, H3A 0B9}
\email{guan@math.mcgill.ca}
\email{siyuan.lu@mail.mcgill.ca}
\thanks{Research of the first author was supported in part by an NSERC Discovery Grant. Research of the second author was supported in part by CSC fellowship and Schulich Graduate fellowship.}
\usepackage{graphicx}
\newtheorem{theorem}{Theorem}
\newtheorem{lemma}{Lemma}

\begin{document}

\begin{abstract}
We establish mean curvature estimate for immersed hypersurface with nonnegative extrinsic scalar curvature in Riemannian manifold $(N^{n+1}, \bar g)$ through regularity study of a degenerate fully nonlinear curvature equation in general Riemannian manifold. The estimate has a direct consequence for the Weyl isometric embedding problem of $(\mathbb S^2, g)$ in $3$-dimensional warped product space $(N^3, \bar g)$. We also discuss isometric embedding problem in spaces with horizon in general relativity, like the Anti-de Sitter-Schwarzschild manifolds and  the Reissner-Nordstr\"om manifolds. \end{abstract}
\maketitle

\section{Introduction}

The paper concerns the regularity of immersed hypersurfaces in Riemannian manifolds. The basic question is how the intrinsic and extrinsic  geometries determine the regularity of the immersion. If the ambient space is $\mathbb R^{3}$, this type of regularity question is related to the classical Weyl problem \cite{W}. The problem was solved by Nirenberg in his landmark paper \cite{N}. Prior to Nirenberg's work \cite{N}, Lewy \cite{L} solved the problem when the metric $g$ is analytic. One of the key steps is the regularity estimates. $C^2$ estimate in \cite{N} for surfaces of positive Gauss curvature has been extended to degenerate elliptic case. For compact surfaces in $\mathbb R^3$ with $K_g\ge 0$, the estimates of principal curvatures were obtained by Guan-Li \cite{GL} and Hong-Zuily \cite{HZ}, see also \cite{I}.  Li and Weinstein \cite{LW} further obtained similar type of estimates for embedded compact convex hypersurfaces $M^n$ in $\mathbb R^{n+1}$ in general dimension $n\ge 2$.
The Weyl isometric embedding problem in hyperbolic space was considered by Pogorelov \cite{P}, we also refer \cite{Pb} and references therein for discussions of isometric embeddings of $(\mathbb S^2, g)$ to general $3$-dimensional Riemannian manifolds. In hyperbolic case, an explicit mean curvature bound was recently proved by Chang and Xiao \cite{CX} for $K_g\ge -1$ under the condition that the set $\{K_g(X)=-1\}$ is finite, and sequentially by Lin and Wang \cite{WL} for general isometric embedded surfaces in $\mathbb H^3$ with $K_g\ge -1$.

\medskip

The primary focus of this paper is the estimate for second fundamental form of immersed hypersurface $(M^n,g)$ in general ambient manifold $(N^{n+1}, \bar g)$ in dimension $n\ge 2$ and application to the Weyl's isometric embedding problem of $(\mathbb{S}^2, g)$ in general $3$-dimensional Riemannian manifolds $(N^3, \bar g)$. The isometric embedding of surfaces plays a prominent role in general relativity. Recent work of Wang-Yau \cite{WY} brought some renewed interest on isometric embedding problem in general $3$-dimentional Riemannian manifolds. The Brown-York quasi-local mass is defined using solution to the classical Weyl problem \cite{BY}. For a two-surface $M$ with positive Gauss curvature bounds a space-like region $\Omega$ in a space-time $\tilde N$. Denote $H_{\Omega}$ to be the mean curvature of $M$ with respect to the outward normal of $\Omega$, and denote $H_{o}$ to be the mean curvature of the isometric embedding of $M$ into $\mathbb R^3$ ({\it solution to the Weyl problem}). The Brown-York mass is defined to be:
\begin{equation}\label{Brown-York} m_{BY}=\frac{1}{8\pi}\int_{M} (H_o-H_{\Omega}) d\sigma.\end{equation}
The positivity of $m_{BY}$ was shown by Shi-Tam \cite{ST}. In \cite{LY1, LY2}, Liu-Yau introduced Liu-Yau quasi-local mass,
\begin{equation}\label{Liu-Yau} m_{LY}=\frac{1}{8\pi}\int_{M} (H_o-|H|) d\sigma,\end{equation}
where $|H|$ is the Lorentzian norm of the mean curvature vector. The positivity of $m_{LY}$ was established in \cite{LY2}. In both of the above definitions for quasi-local masses, the Nirenberg's solution of isometric embedding of $(\mathbb S^2, g)$ to the flat $\mathbb R^3$ is used. Wang-Yau \cite{WY1} generalized Liu-Yau quasi-local mass using Pogorelov's work on isometric embedding of $(\mathbb S^2, g)$ to hyperbolic space $\mathbb H^3$. In a remarkable paper \cite{WY}, Wang-Yau further studied a new quasi-local mass for spacelike 2-surfaces in space time using isometric embedding.  It is clear that quasi-local masses and their positivity are intimately related to the isometric embedding of surfaces. For example, $3$-dimensional Anti-de Sitter-Schwarzschild manifolds can be viewed as slices in $4$-dimensional deSitter-Schwarzschild space-time, understanding of isometric embeddings of $2$-surfaces in these $3$-dimensional manifolds is an important problem.   These ambient spaces are  equipped with warped product structure, where ambient metrics are of form
\begin{equation}\label{warp2} \bar{g}=dr^2+\phi^2(r)d\sigma_{\mathbb S^2}^2, \end{equation} where $\phi(r)$ is defined for $r\ge r_0\ge 0$ and $d\sigma_{\mathbb S^2}^2$ is the standard metric on $\mathbb{S}^2$. $\phi(r)=r,  \phi(r)=\sinh r$ and $\phi(r)=\sin r$ correspond to space form $\mathbb R^3, \mathbb H^3$ and $\mathbb S^3$ respectively.
For general $n\ge 2$, warped product space $(N^{n+1}, \bar g)$ is a Riemannian manifold where $\bar g$ is defined in (\ref{warp2}) with
$d\sigma_{\mathbb S^2}^2$ replaced by $d\sigma_{\mathbb S^n}^2$.

\medskip

Let's fix some notation. Let $(M^n,g)$  be an isometrically immersed hypersurface in an ambient space $(N^{n+1}, \bar g)$ for $n\ge 2$. Denote $Ric$ and $\bar{Ric}$ the Ricci curvature tensors of $(M,g)$  and ($N,\bar{g}$) respectively, and denote $R$ and $\bar{R}$ to be the scalar curvatures of $M$ and $N$ respectively. Fixed a unit normal $\nu$ locally, denote $\kappa_i, i=1,\cdots, n$ to be the principal curvatures of $M$ with respect to $\nu$. Denote $\sigma_2$ the second elementary symmetric function,  
we call $ \sigma_2(\kappa_1,\cdots,\kappa_n)$  the {\bf extrinsic scalar curvature} of the immersed hypersurface. It is clear that it is independent the choice of unit normal $\nu$ as $\sigma_2$ is an even function. 
The Gauss equation yields,
\begin{equation}\label{scalar} \sigma_2(\kappa_1,\cdots,\kappa_n)=\frac{1}{2}(R-\bar{R})+\bar{Ric}(\nu,\nu).\end{equation} 
\medskip

The main result of this paper is the curvature estimate (\ref{estH}) of immersed hypersurafces with nonnegative extrinsic scalar curvature in Riemannian manifolds of general dimensions. If the extrinsic scalar curvature is strictly positive, this estimate and the classical $C^{2,\alpha}$ estimate of Nirenberg \cite{N1} yield desired full regularity for isometrically embedded $(\mathbb S^2, g)$ in $(N^3, \bar g)$. Together with recent result of Li-Wang \cite{GLW} on the solvability of the associated linearized system, existence of isometric embedding can be established.

\begin{theorem}\label{thm1}
Let $(N, \bar g)$ be a warped product space where $\bar g$ defined as in (\ref{warp2}). Denote $\phi^{'}(\rho)=\frac{d \phi}{d \rho}$ and $\Phi(\rho)=\int^{\rho}_0\phi (r)dr$. Suppose $X: (M^n,g)\to (N, \bar g)$ is a $C^4$ immersed compact hypersurface with nonnegative extrinsic scalar curvature and $\phi^{'}>0$ in $M$, then there exists constant $C$ depending only on $n$, $\|g\|_{C^4(M)}$, $\|\bar{g}\|_{C^4(\tilde M)}$ (where $\tilde M$ is any open set in $N$ containing $X(M)$), $\sup_{x\in M}\Phi(X(x))$ and $\inf_{x\in M}\phi^{'}(X(x))$ such that
\begin{equation}\label{estH}
\max_{x\in M,i=1,...,n}|\kappa_i(X(x))|\leq C.
\end{equation}
\end{theorem}

When $(N^{n+1}, \bar g)$ is the standard Euclidean space $\mathbb R^{n+1}$, estimate (\ref{estH}) was proved in \cite{LW} with an explicit constant for embedded hypersurfaces with nonnegative sectional curvature. It was observed in \cite{LW} that scalar curvature equation is the key for curvature estimate for isometric embedding problem in high dimensions. 

If $(N, \bar g)$ is not an Einstein manifold, the right hand side of equation (\ref{scalar}) depends on the normal $\nu$ in a nontrivial way. It is a fully nonlinear equation of the form
\begin{equation}\label{gF}F(\kappa_1,\cdots,\kappa_n)=f(X, \nu)>0.\end{equation}
This is a curvature type equation arising from many classical geometric problems, like the prescribing general Weingarten curvature problem \cite{Alex, CNS, GG}, the problem of prescribing curvature measures in convex geometry \cite{Alex1, GLM, GLL}. For hypersurfaces in $\mathbb R^{n+1}$, if $f$ is independent of normal vector $\nu$, curvature estimate has been obtained by Caffarelli-Nirenberg-Spruck \cite{CNS} for a general class of fully nonlinear operators $F$. When $f$ depends on $\nu$, in contrast to estimates in \cite{CNS}, curvature estimates for solutions of equation (\ref{gF}) are not true in general. For $F=\frac{\sigma_k}{\sigma_l}$ with $n\ge k>l\ge 1$ fixed in $\mathbb R^{n+1}$, there exists a sequence of strictly convex solutions of (\ref{gF}) with unbounded second fundamental forms while the prescribed functions $f(X, \nu)$ in (\ref{gF}) are bounded from below and above  and with uniform $C^3$ bounds (Theorem 2 in  \cite{GRW}). This indicates subtlety of the issue. The curvature estimate was established recently by Guan-Ren-Wang \cite{GRW} for convex hypersurfaces when $F$ in (\ref{gF}) is an elementary symmetric function $\sigma_k, 1\le k\le n$.  It's also proved in \cite{GRW} that,  curvature estimate holds for general starshaped admissible solutions of equation (\ref{gF}) when $F=\sigma_2$.  Spruck-Xiao \cite{SX} subsequently found a very nice simplified proof of the estimate in \cite{GRW} for $F=\sigma_2$, their estimate is also valid in general space form.  On the other hand, estimates obtained in \cite{CNS, GG, GRW, SX} depend also on the lower bound of $f$. 

Estimate (\ref{estH}) does not depend on the lower bound of $\sigma_2(\kappa)$.   In fact, the only requirement in Theorem \ref{thm1} is the extrinsic scalar curvature $\sigma_2(\kappa)\ge 0$, the usual assumption $\kappa \in \bar\Gamma_2$ (definition (\ref{garding2}) in next section) is not imposed. The estimate in Theorem \ref{thm1} yileds $C^{1,1}$ regularity of general immersed hypersurfaces in $(N^{n+1}, \bar g)$ for all $n\ge 2$ with nonnegative extrinsic scalar curvature $\sigma_2(h_{ij})$. The higher regularities of immersed hypersurface will follow if the extrinsic scalar curvature is strictly positive. In dimension $2$, this is a consequence of Nirenberg's work \cite{N1}. For higher dimensions, it follows from Evans-Krylov theorem \cite{E, Kr}. 
\medskip

Estimate (\ref{estH}) is also valid for a general class of ambient spaces.
\begin{theorem}\label{thm2}
Suppose that $X: (M^n,g)\to (N, \bar g)$ is a $C^4$ immersed compact hypersurface with nonnegative extrinsic scalar curvature. Then estimate (\ref{estH}) holds provided that there is $\Phi\in C^2(M)$ such that
\begin{eqnarray}\label{condPhi}  \Phi_{ij}(x)\ge C_1 g_{ij}(x)-C_2h_{ij}(x), \quad \forall x\in M, \end{eqnarray}
for some positive constant $C_1>0, C_2>0$, where $C$ in (\ref{estH}) depend only on $n$, $\|g\|_{C^4(M)}$, $\|\bar{g}\|_{C^4}$,    $\sup_{x\in M}\Phi(X(x))$ and $C_1, C_2$ in (\ref{condPhi}). Besides the warped product space stated in Theorem \ref{thm1}, condition (\ref{condPhi}) is satisfied in each of the following cases: \begin{enumerate}
\item $(N,\bar g)$ is a complete non-compact Riemannian manifold with nonnegative sectional curvature.
\item  There exist $k\ge 0, r>0$ and $p\in N$ such that  $r\le \min\{ inj(p), \frac{\pi}{2\sqrt{k}}\} $,   $M\subset B_{r}(p)$,  and $K_N(X)\leq k, \forall X\in B_{r}(p) $, where $K_N(X)=\sup\{K(e_i,e_j)| \forall e_i,e_j\in T_XN\}$ and $inj(p)$ is the injectivity radius at $p$ and $B_{r}(p)$ is the geodesic ball centred at $p$ of radius $r$.
\end{enumerate}
\end{theorem}

\medskip

The organization of the paper is as follow. In the next section, we obtain a priori bounds for the gradient and the Laplace operator of the extrinsic scalar curvature $\sigma_2(h_{ij}(x))$ in terms of intrinsic and extrinsic geometries.  Section 3 is devoted to the proof of Theorem \ref{thm1} and Theorem \ref{thm2}. The existence of isometric embedding of $(\mathbb S^2, g)$ to $(N^3,\bar g)$, in particular for ambient spaces like the Anti-de Sitter-Schwarzschild manifolds and  the Reissner-Nordstr\"om manifolds,  will be discussed in section 4.

\section{preliminary estimates}

We consider equation (\ref{gF}) as a degenerate fully nonlinear equation where $f$ depends on $\nu$ in non-trivial way.  This section devotes certain derivative estimates on $f$ defined in (\ref{gF}).

Let $(M^n,g)$ be an isometrically immersed  hypersurface  in an ambient Riemannian manifold $(N^{n+1}, \bar g)$. Denote $R_{ijkl}$ and $\bar{R}_{abcd}$ to be the Riemannian curvatures of $M$ and $N$ respectively. For a fixed local frame $(e_1, \cdots, e_n)$ on $M$, let $\nu$ be a normal vector field of $M$, and let $h=(h_{ij})$ be the second fundamental form of $M$ with respect to $\nu$. We have the Gauss equation and Codazzi equation,
\begin{align}\label{Gauss}
R_{ijkl}=\bar{R}_{ijkl}+h_{ik}h_{jl}-h_{il}h_{jk}, \quad (Gauss)
\end{align}
\begin{align}\label{Codazzi}
\nabla_k h_{ij}=\nabla_j h_{ik}+\bar{R}_{\nu ijk} .\quad (Codazzi)
\end{align}
The convention that $R_{ijij}$ denotes the sectional curvature is used here.

The following commutator formulas will be used through out the paper,
\begin{align}\label{comm}
\nabla_i\nabla_jh_{kl}=&\nabla_k\nabla_lh_{ij}-h_{ml}(h_{im}h_{kj}-h_{ij}h_{mk})-h_{mj}(h_{mi}h_{kl}-h_{il}h_{mk})\\ \nonumber
&+h_{ml}\bar{R}_{ikjm}+h_{mj}\bar{R}_{iklm}+\nabla_k\bar{R}_{ijl\nu}+\nabla_i\bar{R}_{jkl\nu},
\end{align}
\begin{align}\label{Curv-deri}
\nabla_i\bar{R}_{jkl\nu}=\bar{\nabla}_i\bar{R}_{jkl\nu}-h_{ij}\bar{R}_{\nu kl\nu}-h_{ik}\bar{R}_{j\nu l\nu}+h_{im}\bar{R}_{jklm}.
\end{align}

Take trace of the Gauss equation,
\begin{eqnarray*}
Ric(i,i)=\bar{Ric}(i,i)-\bar{R}_{i\nu i\nu}+\sum_j\left( h_{ii}h_{jj}-h_{ij}^2\right),\end{eqnarray*}
and the scalar curvature of $M$ is,
\begin{eqnarray*}R=\bar{R}-2\bar{Ric}(\nu,\nu)+2\sigma_2(h).
\end{eqnarray*}
It follows that,
\begin{equation}\label{f-scalar} \sigma_2(h(x)) =f(x, \nu(x)), \forall x\in M,\end{equation}
where
\begin{eqnarray}\label{f-eq} f(x, \nu(x))= \frac{R(x)-\bar{R}(X(x))}2+\bar{Ric}_{X(x)}(\nu(x),\nu(x)).\end{eqnarray}

\medskip

Denote $\mathcal{S}_n$ the collection of $n\times n$ symmetric matrices. Define Garding's $\Gamma_2$ cone as
\begin{equation} \label{garding2}\Gamma_2=\{A\in \mathcal{S}_n | \sigma_1(A)>0, \sigma_2(A)>0.\}\end{equation}
Denote $\bar \Gamma_2$ the closure of $\Gamma_2$.
The following lemma is a special case of Lemma 3.2 in \cite{GLL}, we also deal with degenerate case $\sigma_2=0$.
Here we give a proof using the fact $\sigma_2$ is hyperbolic in $\Gamma_2$ (i.e., $\sigma_2(x+ty)=0$ as a polynomial
of $t \in \mathbb{C}$ has only real roots $\forall x\in \mathbb R, y\in \Gamma_2$). 
\begin{lemma}\label{lemma2}
Let $W(x)=(h_{ij}(x))$ be a 2-symmetric tensor on $M$, suppose that $p\in M$, $W(p)$ is diagonal, $0\le \sigma_2(W(x))\in C^1$ in a neighborhood of point $p$,  and $\sigma_1(W(p))\neq 0$. For each $m=1,...,n$, denote
\[W_m(p)=(\nabla_m h_{11}(p), \cdots, \nabla_m h_{nn}(p)).\]
then at $p$,
\begin{align}
- \sigma_2(W_m,W_m) \geq \min\{-2\frac{\nabla_m\sigma_2(W)\nabla_m\sigma_1(W)}{\sigma_1(W)}+2\frac{(\nabla_m\sigma_1(W))^2\sigma_2(W)}{\sigma_1^2(W)},0\}.
\end{align}
and
\begin{align*}
- \sigma_2(W_m,W_m) \geq \min\{-2\frac{\nabla_m\sigma_2(W)\nabla_m\sigma_1(W) }{\sigma_1(W)}+\frac{(\nabla_m\sigma_2(W))^2\sigma_2(I,I)}{\sigma_2^2(W,I)},0\}.
\end{align*}
\end{lemma}

\begin{proof}
We first prove

\noindent
{\bf Claim:} Suppose that $W, V$ satisfy $\sigma_1(W)\neq 0$, $\sigma_2(W)\ge 0$ and  $\sigma_2(V,W)=0$, then $\sigma_2(V,V)\leq 0$.

We may assume $\sigma_1(W)>0$ by switching $W$ to $-W$ if necessary. The {\it claim} follows from the hyperbolicity of $\sigma_2$ in $\Gamma_2$ (see \cite{G}) if $\sigma_2(W)>0$. The degenerate case $\sigma_2(W)=0$ can be dealt as follow.  Set  $W_{\epsilon}=W+\epsilon I$ and $V_{\epsilon}=V-\frac{\epsilon \sigma_1(V)I}{\sigma_1(W)+\epsilon\sigma_2(I,I)}$.  Since $\sigma_1(W)>0$, $\forall \epsilon>0$,   $W_{\epsilon}\in \Gamma_2$ and $\sigma_2(W_{\epsilon}, V_{\epsilon})=0$. By the hyperbolicity of $\sigma_2$ in $\Gamma_2$,  $\sigma_2(V_{\epsilon},V_{\epsilon})\leq 0$. The {\it claim} follows by taking $\epsilon\to 0$. Note that the {\it claim} may not be true if the condition $\sigma_1(W)\neq 0$ is dropped. If $\sigma_1(W)=0$, as $\sigma_2(W)\ge 0$, we must have $\sigma_2(W)=0$ which in turn implies $W=0$. So $\sigma_2(W,V)=0, \forall V$. In particular any $V\in \Gamma_2$ would violate the {\it claim}.

Now back to the proof of the lemma. Denote  $W_m=(\nabla_mh_{ii})$ and $\nabla_{m}\sigma_2(W)=(\sigma_2(W))_m$.
If $\sigma_2(W(p))=0$, since $\sigma_2(W(x))\ge 0$ near $p$ and $W$ is diagonal at $p$, we have $0=(\sigma_2(W))_m=\sigma_2(W_m,W)$ at $p$. By the assumption and the {\it claim}, $\sigma_2(W_m,W_m)\leq 0$ at $p$.  

If $\sigma_2(W(p))>0$, we have $W(p)\in \Gamma_2$. Set, $V=W_m -\frac{\sigma_2(W,W_m)}{\sigma_2(W,I)}I$. So, $\sigma_2(W, V)=0$. By Garding \cite{G}
$\sigma_2(V,V)\le 0$, that is ,
\begin{eqnarray*} 0 \ge \sigma_2(V,V)= \sigma_2(W_m,W_m)-2\frac{\sigma_2(W,W_m)\sigma_2(W_m,I)}{\sigma_2(W,I)
}+\frac{\sigma_2^2(W,W_m)\sigma_2(I,I)}{\sigma_2^2(W,I)}.\end{eqnarray*}

In turn,
\begin{eqnarray*} - \sigma_2(W_m,W_m) &\ge&  -2\frac{\sigma_2(W,W_m)\sigma_2(W_m,I)}{\sigma_2(W,I)}+\frac{\sigma_2^2(W,W_m)\sigma_2(I,I)}{\sigma_2^2(W,I)}\\
&=&-2\frac{\nabla_m\sigma_2(W)\nabla_m\sigma_1(W)}{\sigma_1(W)}+\frac{(\nabla_m\sigma_2(W))^2\sigma_2(I,I)}{\sigma_2^2(W,I)}.\end{eqnarray*}
as $\sigma_2(W,W_m)=\nabla_m\sigma_2(W)$, $\sigma_2(W_m,I)=(n-1)\nabla_m\sigma_1(W)$ and $\sigma_2(W,I)=(n-1)\sigma_1(W)$.

This fulfills the second inequality. Now let's prove the first inequality. At point $p$, If $\sigma_1(W_m)=0$, then $\sigma_2(W_m,W_m)\leq 0$. Suppose now $\sigma_1(W_m)\neq 0$, let $V=W_m-\frac{\sigma_1(W_m)}{\sigma_1(W)}W$, then $\sigma_1(V)=0$, thus $\sigma_2(V,V)\leq 0$, i.e.
\begin{align*}
0\geq \sigma_2(V,V)=\sigma_2(W_m,W_m)-2\frac{\sigma_1(W_m)\sigma_2(W_m,W)}{\sigma_1(W)}+\frac{\sigma_1^2(W_m)\sigma_2(W,W)}{\sigma_1^2(W)}.
\end{align*}
In turn,
\begin{align*}
-\sigma_2(W_m,W_m)\geq -2\frac{\nabla_m\sigma_1(W)\nabla_m\sigma_2(W)}{\sigma_1(W)}+2\frac{(\nabla_m\sigma_1(W))^2\sigma_2(W)}{\sigma_1^2(W)}.
\end{align*}
The lemma is now proved.
\end{proof}

\begin{lemma}\label{lemma5}
\begin{align}\label{formula}
|\Delta_g f(x)|\leq C(\sum_{i,j} |h_{ij}(x)|^2+|\nabla H|+1),
\end{align}
for any $x\in M$, where $C$ depends on $\|g\|_{C^4}$ and $\|\bar{g}\|_{C^4}$.
\end{lemma}

\begin{proof}
$\forall x_0\in M\subset N$, fix a local orthonormal coordinates $(x_1, \cdots, x_n)$ at $x_0\in M$,  a local orthonormal coordinates $(X_1, \cdots, X_{n+1})$ of $x_0\in N$. View $\bar{Ric}$ as a function in $C^2(N\times \mathbb R^{n+1}\times \mathbb R^{n+1})$ locally, denote $\bar{Ric}_\alpha=\frac{\partial\bar{Ric}}{\partial X_\alpha}$. For each $X$ fixed, $\bar{Ric}(\xi, \eta)$ is a bilinear function of $\xi, \eta \in \mathbb R^{n+1}$.

Denote $X_{i}^{\alpha}=\frac{\partial X^\alpha}{\partial x_i}, X_{ii}^{\alpha}=\frac{\partial^2 X^\alpha}{\partial x_i^2}$, we have
\begin{align*}
f_i=\frac{1}{2}\left(R_i-\bar{R}_\alpha X_{i}^{\alpha}\right)+\bar{Ric}_\alpha (\nu,\nu)X_{i}^{\alpha}+2\bar{Ric}\left(\frac{\partial \nu}{\partial x_i},\nu\right),
\end{align*}
and
\begin{eqnarray*}
f_{ii}&=&\frac{R_{ii}-\bar{R}_{\alpha\beta}X_{i}^{\alpha}X_{i}^{\beta}-\bar{R}_\alpha X_{ii}^{\alpha}}{2}+2\bar{Ric}\left(\frac{\partial^2 \nu}{\partial x_i^2},\nu\right)+2\bar{Ric}\left(\frac{\partial \nu}{\partial x_i},\frac{\partial \nu}{\partial x_i}\right)\\
& &+\bar{Ric}_{\alpha} (\nu,\nu)X_{ii}^{\alpha}+4\bar{Ric}_\alpha\left(\frac{\partial \nu}{\partial x_i},\nu\right)X_{i}^{\alpha}+\bar{Ric}_{\alpha\beta} (\nu,\nu)X_{i}^{\alpha}X_{i}^{\beta}.
\end{eqnarray*}

Since
\begin{align*}
\frac{\partial \nu}{\partial x_i}=h_{ij}e_j,\quad \frac{\partial^2 \nu}{\partial x_i^2}=h_{iji}e_j-h^2_{ij}\nu,
\end{align*}
and
\begin{align*}
|\frac{\partial X^\alpha}{\partial x_i}|\leq C,\quad |\frac{\partial^2 X^\alpha}{\partial x_i^2}|\leq C(\sum_j|h_{ij}|+1).
\end{align*}
Thus, by Codazzi equation
\begin{eqnarray}\label{formula1}
f_{ii} & = & 2\sum_{jk} h_{ij}h_{ik}\bar{Ric}(j,k)-2\sum_j h_{ij}^2\bar{Ric}(\nu,\nu)+2\sum_j h_{iij}\bar{Ric}(j,\nu) \nonumber \\
&&-O(\sum_j |h_{ij}|+1).
\end{eqnarray}
Sum over $i$, (\ref{formula}) follows directly.
\end{proof}

\section{Proof of theorems}

Suppose $(N, \bar g)$ is a warped product space with an ambient metric $\bar g$ as
\begin{equation}
\bar{g}=d\rho^2+\phi^2(\rho)ds_{\mathbb S^n}^2
\end{equation}
where $ds_{\mathbb S^n}^2$ is the standard induced metric in $\mathbb{S}^n$, $\rho$ represents the distance from the origin.
The vector field $V=\phi(\rho)\frac{\partial}{\partial\rho}$ is a conformal Killing field in $N$.
For $\Phi(\rho)=\int^\rho_0\phi(r)dr$,
the following well-known fact (e.g. \cite{G-L}) will be used.
\begin{lemma}\label{lemma1} Let $(e_1,\cdots,e_n)$ be a local orthonormal frame of an oriented Riemannian manifold $(M^n,g)$  which is immersed as a hypersurface in $N$. Let $\Phi$ and $V$ defined as above. Let $\nu$ be a given unit normal and $h_{ij}$ be the second fundamental form of the hypersurface with respect to $\nu$. Then $\Phi|_M$ satisfies,
\begin{equation}
\nabla_i\nabla_j\Phi=\phi^\prime(\rho)g_{ij}-<V,\nu>h_{ij}
\end{equation}
where $\nabla$ is the covariant derivative with respect to $g$.
\end{lemma}

Theorem \ref{thm1} follows from the explicit bound in the next theorem.

\begin{theorem}\label{thm1-1}
Let $(N, \bar g)$ be a warped product space where $\bar g$ defined as in (\ref{warp2}). Denote $\phi^{'}(\rho)=\frac{d \phi}{d \rho}$ and $\Phi(\rho)=\int^{\rho}_0\phi (r)dr$. Suppose $X: (M^n,g)\to (N, \bar g)$ is a $C^4$ immersed compact hypersurface with nonnegative extrinsic scalar curvature and $\phi^{'}>0$ in $M$, then there exist constants $A_1, A_2$ depending only on $n$, $\|g\|_{C^4}$ and $\|\bar{g}\|_{C^4}$ such that
\begin{equation}\label{estH-1}
\max_{x\in M,i=1,...,n}|\kappa_i(X(x))|\leq A_1\exp\left( A_2\frac{\sup_{x\in M} \Phi(X(x))-\inf_{x\in M} \Phi(X(x))}{\inf_{x\in M} \phi^\prime (X(x))} \right).
\end{equation}
\end{theorem}

\begin{proof}
Denote by $\kappa(x)=(\kappa_1(x), \cdots, \kappa_n(x))$ the principal curvatures of $x\in M$. 
Set,
\begin{equation*}
\varphi=log|H|+\alpha \frac{\Phi}{m},
\end{equation*}
where $H=\sigma_1(h)$ is the mean curvature, $m=\inf_{x\in M}\phi^\prime(X(x))$ and $\alpha$ is a positive constant to be determined later. Suppose $\varphi$ attains maximum at $x_0$. Without loss of generality, we may assume $|H|(x_0)\geq 1$, otherwise there's nothing to prove. With a suitable choice of local orthonormal frame
$(e_1,\cdots,e_n)$, we may also assume $H(x_0)\geq 1$ and $h_{ij}(x_0)$ is diagonal so that $\kappa_i=h_{ii}$.

In the rest of proof, all computations will be carried out at $x_0$.
\begin{equation}\label{3.2}
\varphi_i=\frac{\sum_l h_{lli}}{H}+\alpha \frac{\Phi_i}{m}=0,
\end{equation}
\begin{eqnarray}\label{3.3}
\varphi_{ii}=\frac{\sum_l h_{llii}}{H}-\frac{(\sum_lh_{lli})^2}{H^2}+\alpha \frac{\Phi_{ii}}{m}.
\end{eqnarray}
By commutator formula (\ref{comm}),
\begin{equation}\label{3.4}
h_{llii}=h_{iill}-h_{ii}^2h_{ll}+h_{ll}^2h_{ii}+h_{ll}\bar{R}_{ilil}+h_{ii}\bar{R}_{illi}+\nabla_l\bar{R}_{iil\nu}+\nabla_i\bar{R}_{ill\nu}.
\end{equation}

Put (\ref{3.4}) into (\ref{3.3}), at $x_0$,
\begin{eqnarray}\label{3.7}
\sigma_2^{ii}\varphi_{ii}&=& \sum_l\frac{ \sigma_2^{ii}\left(h_{iill}-h_{ii}^2h_{ll}+h_{ll} \bar{R}_{ilil}+h_{ii}\bar{R}_{illi}
 +\nabla_l\bar{R}_{iil\nu}+\nabla_i\bar{R}_{ill\nu}\right)}{H} \nonumber \\
 &&
+\frac{2fh_{ll}^2}{H}-\frac{\sigma_2^{ii}(\sum_l h_{lli})^2}{H^2}+\alpha \frac{\sigma_2^{ii}\Phi_{ii}}{m}.
\end{eqnarray}

Differentiate equation (\ref{f-scalar}) in direction of $e_k$,
\begin{eqnarray*}
\sigma_2^{ii}h_{iik}=f_k, \quad
\sigma_2^{ii}h_{iikk}+\sigma_2^{pq,rs}h_{pqk}h_{rsk}=f_{kk}.
\end{eqnarray*}
Put above identities into inequality (\ref{3.7}),
\begin{eqnarray*}
\sigma_2^{ii}\varphi_{ii}&=&\sum_l\frac{ \sigma_2^{ii}\left(h_{ll}\bar{R}_{ilil}+h_{ii}\bar{R}_{illi}
+\nabla_l\bar{R}_{iil\nu}+\nabla_i\bar{R}_{ill\nu}\right)-\sigma_2^{pq,rs}h_{pql}h_{rsl}}{H}\\
&&-\sigma_2^{ii}h_{ii}^2-\frac{\sigma_2^{ii}(\sum_l h_{lli})^2}{H^2}+\alpha \frac{\sigma_2^{ii}\Phi_{ii}}{m}+\frac{2fh_{ll}^2}{H}+ \frac{f_{ll}}{H}.
\end{eqnarray*}
By (\ref{Curv-deri}),  \[|\nabla_l\bar{R}_{iil\nu}+\nabla_i\bar{R}_{ill\nu} |\le CH.\] As $0\le \sigma_2^{ii} \le CH$, at the maximum point $x_0$,
\begin{align}\label{ineqq}
0\geq \frac{\sum_l\left( f_{ll}-\sigma_2^{pq,rs}h_{pql}h_{rsl}\right)}{H}-\sigma_2^{ii}h_{ii}^2-\frac{\sigma_2^{ii}(\sum_l h_{lli})^2}{H^2}+\alpha \frac{\sigma_2^{ii}\Phi_{ii}}{m}-CH,
\end{align}
where $C$ is a constant depending on $n$, $\|g\|_{C^4}$, $\|\bar{g}\|_{C^4}$. In what follows, we denote $C$ as a constant, which might change from line to line, but it always stands for a constant under control.

By Lemma \ref{lemma1},
\begin{equation}\label{3.5}
\Phi_{ii}=\phi^\prime(\rho)-h_{ii}u.
\end{equation}
Put this to (\ref{ineqq}),
\begin{eqnarray}\label{ineq}
0&\geq & \frac{1}{H}\left(\sum_l\left( f_{ll}-\sigma_2^{pq,rs}h_{pql}h_{rsl}\right)\right)-\sigma_2^{ii}h_{ii}^2-\frac{\sigma_2^{ii}(\sum_l h_{lli})^2}{H^2}\nonumber \\
& &+(\alpha-C)H-C\frac{\alpha\phi}{m},
\end{eqnarray}
where $C$ is a constant depending on $n$, $\|g\|_{C^4}$ and $\|\bar{g}\|_{C^4}$.

By Lemma \ref{lemma5} and (\ref{3.2}) and the assumption $H\geq 1$,
\begin{eqnarray}\label{3.8}
\quad \quad \quad 0\geq (\alpha-C)H -\frac{\sum_l\sigma_2^{pq,rs}h_{pql}h_{rsl}}{H}-\sigma_2^{ii}h_{ii}^2
-\frac{\sigma_2^{ii}(\sum_l h_{lli})^2}{H^2}
-C\frac{\alpha\phi}{m}.
\end{eqnarray}

Note that
\begin{align}\label{3.9}
-\sigma_2^{pq,rs}h_{pql}h_{rsl}=-\sum_{p\neq q}(h_{ppl}h_{qql}-h_{pql}^2),
\end{align}
It follows from Lemma \ref{lemma2} that,
\begin{equation*}
-\sum_{p\neq q}h_{ppl}h_{qql}\geq \min\{-2\frac{(\sigma_2)_l(\sigma_1)_l}{\sigma_1},0\}
\end{equation*}
Together with critical condition (\ref{3.2}) and the defintion of $\sigma_2$, we have
\begin{align}\label{3.10}
-\sum_{p\neq q}h_{ppl}h_{qql}\geq -C\frac{\alpha\phi}{m}H
\end{align}

Put (\ref{3.9}) and (\ref{3.10}) into (\ref{3.8})
\begin{eqnarray*}
0\ge  \frac{\sum_{p\neq q}h_{pql}^2}{H}-\sigma_2^{ii}h_{ii}^2-\frac{\sigma_2^{ii}(\sum_l h_{lli})^2}{H^2}+(\alpha-C)H-C\frac{\alpha\phi}{m}.
\end{eqnarray*}

Note that
\begin{equation*}
\sigma_2^{ii}h_{ii}^2=\sum_{i=1}^{n}h_{ii}(\sum_{j\neq i}h_{jj}h_{ii}),
\end{equation*}
and $\sum_{j\neq i}h_{jj}h_{ii}$ is bounded by Gauss equation (\ref{Gauss}). Thus,
\begin{equation*}
\sigma_2^{ii}h_{ii}^2\leq C\sum_{i=1}^{n}|h_{ii}|\leq CnH.
\end{equation*}
Therefore,
\begin{eqnarray}\label{3.12}
0\ge \frac{1}{H}\sum_{p\neq q}h_{pql}^2-\frac{\sigma_2^{ii}(\sum_l h_{lli})^2}{H^2}+(\alpha-C)H-C\frac{\alpha\phi}{m}.
\end{eqnarray}

Now we state a lemma, which will be used frequently in the rest of the paper (the same trick was also used in \cite{GLX}).
\begin{lemm}\label{lemm 3.1}
Suppose the second fundamental form $(h_{ij})$ is diagonalized at $x_0$, assume $h_{11}\geq h_{22}\cdots\geq h_{nn}$ and $\sigma_1\geq 0$, then either $H\leq 1$ or $|h_{ii}|\leq \frac{C}{h_{11}}$ for $i\neq 1$, where $C$ is a constant depending only on $\|g\|_{C^2},\|\bar{g}\|_{C^2}$.
\end{lemm}
\begin{proof}
Suppose that $H> 1$, then $h_{11}\geq \frac{H}{n}\geq \frac{1}{n}$. By Gauss equation (\ref{Gauss}),  $|h_{11}h_{ii}|=|R_{1i1i}-\bar{R}_{1i1i}|\leq C$, we deduce that $|h_{ii}|\leq \frac{C}{h_{11}}$.
\end{proof}

By lemma \ref{lemm 3.1}, $\sigma_2^{11}\leq \frac{C}{H}$.
Critical equation (\ref{3.2}) yields
\begin{equation*}
\varphi_i=\frac{H_i}{H}+\alpha\frac{\Phi_i}{m}=0.
\end{equation*}
We have
\begin{equation*}
\sigma_2^{11}\left(\frac{H_1}{H}\right)^2\leq C\frac{\alpha^2\phi^2 }{Hm^2}.
\end{equation*}
By Codazzi equation (\ref{Codazzi}),
\begin{eqnarray}\label{3.13}
0 &\ge & \frac{\sum_{p\neq q}h_{pql}^2}{H}-\sum_{i\neq 1}\frac{\sigma_2^{ii}(\sum_l h_{lli})^2}{H^2}\nonumber \\
& &+(\alpha-C)H-C\frac{\alpha\phi}{m}-C\frac{\alpha^2\phi^2 }{Hm^2} \\
&\ge & \frac{\sum_{l\neq i}2h_{lli}^2}{H}-\sum_{i\neq 1}\frac{\sigma_2^{ii}\left(\sum_l h_{lli}^2+\sum_{p\neq q} h_{ppi}h_{qqi}\right)}{H^2} \nonumber \\
&& +(\alpha-C)H-C\frac{\alpha\phi}{m}-C\frac{\alpha^2\phi^2 }{Hm^2} . \nonumber
\end{eqnarray}
It follows from (\ref{3.10}) that,
\begin{align}\label{3.14}
-\sum_{i\neq 1}\frac{\sigma_2^{ii}\left(\sum_{p\neq q} h_{ppi}h_{qqi}\right)}{H^2}\geq -C\sum_{i\neq 1}\sigma_2^{ii}\frac{\alpha\phi}{mH}\geq -C\frac{\alpha\phi}{m}.
\end{align}

Insert (\ref{3.14}) into (\ref{3.13}),
\begin{align*}
0 &\ge \frac{1}{H}\sum_{l\neq i}2h_{lli}^2-\sum_{i\neq 1}\frac{\sigma_2^{ii} h_{lli}^2}{H^2}+(\alpha-C)H-C\frac{\alpha\phi}{m}-C\frac{\alpha^2\phi^2 }{Hm^2} .
\end{align*}

By Lemma \ref{lemm 3.1}, $\sigma_2^{ii}\leq H+\frac{C}{H}$ for $i\neq 1$, we have
\begin{eqnarray*}
0&\ge& \frac{1}{H}\sum_{l\neq i}(1-\frac{C}{H^2})h_{lli}^2-\sum_{i\neq 1}\frac{(1+\frac{C}{H^2}) h_{iii}^2}{H}\\
& &+(\alpha-C)H-C\frac{\alpha\phi}{m}-C\frac{\alpha^2\phi^2 }{Hm^2} .
\end{eqnarray*}

We deal with $\frac{h^2_{iii}}{H}$. Again by Gauss equation (\ref{Gauss}),
\begin{equation*}
h_{11i}h_{ii}+h_{11}h_{iii}=R_{1i1i,i}-\bar{R}_{1i1i,i}
\end{equation*}
Thus $\forall i\neq 1$,
\begin{equation}\label{n1}
\frac{h_{iii}^2}{H}\leq \frac{2h_{ii}^2h_{11i}^2+C}{Hh_{11}^2}\leq C\frac{h_{11i}^2}{H^5}+\frac{C}{H^3}
\end{equation}
In turn,
\begin{equation} \label{3.15}
0\ge \frac{1}{H}\sum_{l\neq i}(1-\frac{C}{H^2})h_{lli}^2+(\alpha-C)H-C\frac{\alpha\phi}{m}-C\frac{\alpha^2\phi^2 }{Hm^2} .
\end{equation}

Choose $\alpha$ big enough, we have $H\leq \frac{C \phi}{m}$ at the maximum point of $\varphi$. Since $\Phi(x_0)-\min\Phi\geq C\phi(\tilde{\rho})$, we have $H\leq Ce^{\frac{C}{m} \left(\Phi(x_0)-\min\Phi\right)}$.

As \[\sum_i \kappa_i^2=H^2(\kappa)-2\sigma_2(\kappa),\] we obtain a bound on the principal curvatures. The proof of Theorem \ref{thm1-1} is complete. \end{proof}

\medskip

We remark that Theorem \ref{thm1-1} does not make the assumption that $(h_{ij})\in \bar \Gamma_2$. That is, it also allows $(-h_{ij})\in \bar \Gamma_2$ at some points. The only assumption is the nonnegativity of the extrinsic scalar curvature $\sigma_2(h_{ij})$. One observes that, nonnegativity assumption of $\sigma_2(h_{ij}(x))$ and the Newton-MacLaurin inequality imply either  $(h_{ij}(x))\in \bar \Gamma_2$ or
$(-h_{ij})(x)\in \bar \Gamma_2$. 

\bigskip

We now prove theorem \ref{thm2}.

\begin{proof}  Note that the only place where property (\ref{3.5}) of the warp potential $\Phi$ is used in the proof of Theorem \ref{thm1} is to obtain inequality (\ref{ineq}). For this purpose, existence of a function $\Phi$ satisfying inequality (\ref{condPhi}) would suffice. For any manifold with a function $\Phi$ satisfying (\ref{condPhi}), the same proof of Theorem \ref{thm1} will carry through to obtain the curvature bound.
Therefore, we only need to verify that for each manifold listed in Theorem \ref{thm2}, there is a globally defined function $\Phi$ on $M$ such that (\ref{condPhi}) is satisfied.

\begin{enumerate} \item  $(N,\bar g)$ is a complete non-compact non-negatively curved manifold, by \cite{GW2}, there is a strictly global convex function $\Phi$ in $N$. Since $M$ is compact, we may find an open set $U\subset N$ such that $M\subset U$ and $\bar U$ is compact. Therefore, there is $C>0$,
\[\Phi_{\alpha \beta }(X)\ge C\bar g_{\alpha\beta}(X), \quad \forall X\in U, \]
Restricting $\bar \nabla^2\Phi$ to $M$,  \begin{eqnarray*}\Phi_{ij}&\ge& Cg_{ij}-h_{ij}\langle \nabla \Phi,\nu\rangle \\
&\ge & Cg_{ij}-C_2h_{ij},\end{eqnarray*} for some positive constant $C_2>0$.

\item  $K(N)\leq k$, for some $k\ge 0$, $M\subset B_{R}(p)$ with $ R\le \min\{inj(p),\frac{\pi}{2\sqrt{k}}\}$.

Set $\rho(x)=d(p, x)$ and $\Phi(x)=\rho^2(x)$.
Let's recall the definition of the segment domain
\begin{eqnarray*} seg(p)=\{v\in T_pN |  \exp_p(tv): [0,1] \to N \mbox{is a segment} \}.\end{eqnarray*}
Note that $B_R\subset B_{inj(p)}(p)\subset seg^0(p)$, where $seg^0(p)$ is the interior of $seg(p)$. $seg^0(p)$ is a starshaped open domain where $exp_p$ is injective, non-singular and $\Phi$ is smooth.

\medskip

\noindent
We divide it into two cases: $k=0$ and $k>0$.

If $k=0$,   for any $e$ unit tangent vector field orthogonal to $\frac{\partial}{\partial \rho}$ in $TN$,
by Hessian comparison \cite{SY}, compared with $\mathbb{R}^{n+1}$,
\begin{align*}
\rho_{e e}\geq |X|_{\tilde{e}\tilde{e}},
\end{align*}
where $\tilde{e}$ is the corresponding unit tangent vector in $\mathbb{R}^{n+1}$ orthogonal to $\frac{\partial}{\partial r}$.
As $Hess (\rho) (e, \frac{\partial}{\partial \rho})=0$, we have
\begin{align*}
(\Phi_{\alpha\beta}) \geq C_0(\bar g_{\alpha\beta}).
\end{align*}

If $k>0$, by a dilation of the metric, we only need to consider the case $K(N)\leq 1$, $R\le \min\{inj(p), \frac{\pi}{2}\}$. Since $M$ is compact, $M\subset B_{\tilde R}(p)$ for some $\tilde R<R\le \frac{\pi}{2}$.
Again by Hessian comparison,  for any $e$ unit tangent vector field orthogonal to $\frac{\partial}{\partial \rho}$ in $TN$
\begin{align*}
\rho_{e e}\geq r_{\tilde{e}\tilde{e}},
\end{align*}
where $r$ is the distance function on $\mathbb{S}^{n+1}$. If we denote the metric on $\mathbb{S}^{n+1}$ as
\begin{align*}
ds^2=dr^2+\sin^2(r)g_r,
\end{align*}
where $g_r$ the standard round metric on $\mathbb S^n$. We have
\begin{align*}
r_{\tilde{\alpha}\tilde{\beta}}=Hess (r)(e_{\tilde{\alpha}},e_{\tilde{\beta}})=\frac{2\cos(r)}{\sin(r)}g_r(e_{\tilde{\alpha}},e_{\tilde{\beta}}).
\end{align*}
thus
\begin{align*}
\Phi_{\alpha\beta}\geq 2\rho_\alpha\rho_\beta+\frac{4\rho\cos(\rho)}{\sin(\rho)}g_r(e_{\tilde{\alpha}},e_{\tilde{\beta}}).
\end{align*}
Since $\rho\le \tilde R<\frac{\pi}{2}$,
\begin{align*}
\Phi_{\alpha\beta}(x) \geq C_0 \bar g_{\alpha\beta}(x), \quad x\in B_{\tilde R}(p),
\end{align*}
for some constant $C_0>0$.

In both cases, restricting $\bar \nabla^2 \Phi$ to $M$, we have
\begin{align*}
\Phi_{ij}\geq C_0g_{ij}-C_2h_{ij}.
\end{align*}
with some constant $C_2>0$.

\end{enumerate}
\end{proof}

If $N$ is an even-dimensional, compact,
simply connected, with positive sectional curvature $0<K_N\le k$, the Klingenberg's injectivity radius estimate \cite{K} states that $inj(p)\ge \frac{\pi}{\sqrt{k}}, \forall p\in N$. In this case, $\min\{inj(p), \frac{\pi}{2\sqrt{k}}\}=\frac{\pi}{2\sqrt{k}}$.

\begin{rem} If $N$ is a Hadamard space, condition (2) in Theorem \ref{thm2} is automatically satisfied for any $p\in N$ and $r>0$.  When $n=2$ and under the assumption that there is $\delta>0$ such that \begin{eqnarray*}
R(x) \ge \delta+\bar{R}(X(x))-2\inf\{\bar{Ric}_{X(x)}(\mu,\mu):\mu \in T_{X(x)}U,|\mu|=1\}, \forall x\in \mathbb S^2,\end{eqnarray*} estimate (\ref{estH}) was proved by Pogorelov (Theorem in page 401, \cite{Pb}) for embedded strictly convex surfaces $(\mathbb S^2, g)$ in $3$-dimensional Hadamard space, where the constant $C$ in (\ref{estH}) in addition depends on $\delta$ and inner radius of the domain enclosed by the embedded surface. A lower bound of inner radius of convex surfaces in $\mathbb R^3$ with extrinsic Gauss curvature bounded between two positive constants $K_0, K_1$ follows from a classical result of Blaschke. Such lower bound for inner radius was claimed in \cite{Pb} for bounded convex surfaces in general $3$-dimensional Hadamard space.
\end{rem}

\medskip

Estimate (\ref{estH}) in Theorem \ref{thm1} depends on the position of the embedding. It is desirable to replace it by intrinsic diameter of $M$. In the case of space forms, one may achieve this by shifting the origin. From the proof of Theorem \ref{thm2}, we may replace it by distance function in some cases listed there. We state it as a corollary.

\begin{coro} \label{cor-3} If $X: (M^n,g)\to B_{R}(p)\subset  (N, \bar g)$ is a $C^4$ immersed compact oriented hypersurface with nonnegative extrinsic scalar curvature.  Suppose for some $k\ge 0$, $K_N(X)\leq k, \forall x\in B_{R}(p)$ and $R\le \min\{inj(p), \frac{\pi}{2\sqrt{k}} \}$,  then there's a constant $C$ depending only on $n$, $\|g\|_{C^4}$, $\|\bar{g}\|_{C^4}$, $\frac{\pi}{2\sqrt{k}}-\sup_{x\in M}dist(x, p)$ and $dist(M, \partial B_{R}(p))$ such that
\begin{equation}\label{estH-2}
\max_{X\in M,i=1,...,n}|\kappa_i(X)|\leq C.
\end{equation}
\end{coro}
\begin{proof} In the proof of Theorem 2, we used $\Phi=\rho^2$ in cases (2) and (3), where $\rho$ is the distance function of the ambient space $N$.  Then the Corollary follows directly from the proof of Theorem \ref{thm2}. \end{proof}

\medskip

Condition (\ref{condPhi}) in Theorem \ref{thm2} may not be valid in general. For example,
let $M=\mathbb S^n\subset \mathbb S^{n+1}$. Since $h_{ij}=0$ for $M$, if there is $\Phi$ satisfying (\ref{condPhi}) in a neighborhood of $\mathbb S^n$, one would have
\[\Delta \Phi(x)\ge C>0, \forall x\in \mathbb S^n.\]
This is impossible. On the other hand, one may obtain a direct mean curvature estimate for $(M,g)$ when sectional curvature of $(N,\bar g)$  is positive and sufficiently pinched. In particular, if $(N,\bar g)$ is an Einstein manifold with positive sectional curvature.

Assume that $\max |H|$ is attained at some point $x_0$. With a suitable choice of local orthonormal frame $(e_1,\cdots,e_n)$, we may assume that $H(x_0)\geq 1$ and $h_{ij}(x_0)$ is diagonal. At $x_0$,
\begin{align}\label{a}
0&\geq \sigma_2^{ii} H_{ii}=\sum_{l}\sigma_2^{ii} h_{llii}\\\nonumber
&=\sum_{l}\sigma_2^{ii}\left( h_{iill}-h_{ii}^2h_{ll}+h_{ll}^2h_{ii}+h_{ll}\bar{R}_{ilil}+h_{ii}\bar{R}_{illi}+\nabla_l\bar{R}_{iil\nu}+\nabla_i\bar{R}_{ill\nu}\right)\\\nonumber
&=\sum_{l}\left( f_{ll}-\sigma_2^{pq,rs}h_{pql}h_{rsl}\right)-\sigma_1\sigma_2^{ii}h_{ii}^2+2\sigma_2\sum_{l}h_{ll}^2\\\nonumber
&+\sum_{l}\sigma_2^{ii}\left( h_{ll}\bar{R}_{ilil}-h_{ii}\bar{R}_{ilil}+\nabla_l\bar{R}_{iil\nu}+\nabla_i\bar{R}_{ill\nu}\right).
\end{align}
As $\sigma_2^{ii}h_{ii}^2=\sigma_1\sigma_2-3\sigma_3$,
\begin{eqnarray}\label{b1}
0&\ge & \sum_{l}\left( f_{ll}-\sigma_2^{pq,rs}h_{pql}h_{rsl}\right)+3\sigma_1\sigma_3-4\sigma_2^2+\sigma_1^2\sigma_2\\
&& +\sum_{l}\sigma_2^{ii}\left(h_{ll}\bar{R}_{ilil}-h_{ii}\bar{R}_{ilil}+\nabla_l\bar{R}_{iil\nu}
+\nabla_i\bar{R}_{ill\nu}\right).\nonumber
\end{eqnarray}

By lemma \ref{lemm 3.1},
\begin{align*}
|\sigma_3|=|h_{11}\sum_{i,j\neq 1}h_{ii}h_{jj}+\sum_{i,j,k\neq 1}h_{ii}h_{jj}h_{kk}|\leq Ch_{11}^{-1}+Ch_{11}^{-3},
\end{align*}
thus
\begin{align}\label{b2}
\sigma_1\sigma_3\geq -C-Ch_{11}^{-2}\geq -C.
\end{align}
Insert (\ref{b2}) into (\ref{b1}), by the boundedness of $\sigma_2$,
\begin{align}\label{c}
0&\geq \sum_{l}\left( f_{ll}-\sigma_2^{pq,rs}h_{pql}h_{rsl}+\sigma_2^{ii}[(h_{ll}-h_{ii})\bar{R}_{ilil}+\nabla_l\bar{R}_{iil\nu}+\nabla_i\bar{R}_{ill\nu}]\right)-C.
\end{align}
We now work out terms in (\ref{c}). By (\ref{Curv-deri}),
\begin{align}
\nabla_l\bar{R}_{iil\nu}+\nabla_i\bar{R}_{ill\nu}=&\bar{\nabla}_l\bar{R}_{iil\nu}-h_{li}\bar{R}_{\nu il\nu}-h_{li}\bar{R}_{i\nu l\nu}+h_{lm}\bar{R}_{i ilm}\\\nonumber
& +\bar{\nabla}_i\bar{R}_{ill\nu}-h_{ii}\bar{R}_{\nu ll\nu}-h_{il}\bar{R}_{i\nu l\nu}+h_{im}\bar{R}_{illm}\\\nonumber
=&\bar{\nabla}_l\bar{R}_{iil\nu}+\bar{\nabla}_i\bar{R}_{ill\nu}+h_{ii}\bar{R}_{illi}+h_{ii}\bar{R}_{l\nu l\nu} (1-\delta_{il}).
\end{align}
Since $\bar{\nabla}_l\bar{R}_{iil\nu}$ and $\bar{\nabla}_i\bar{R}_{ill\nu}$ are bounded, and by lemma \ref{lemm 3.1}
\[|\sigma_2^{ii}h_{ii}|\le C, \forall i=1,\cdots, n.\]
By (\ref{formula1}) and the fact $\nabla H=0$ at the point,
\begin{align*}
 \sum_i f_{ii}\geq   \sum_i\left(-2h_{ii}^2\bar{Ric}(\nu,\nu)+2h_{ii}^2\bar{Ric}(i,i)-C|h_{ii}|-C\right).
\end{align*}
Again, as $\nabla H=0$,
\begin{eqnarray*}
-\sum_{l}\sigma_2^{pq,rs}h_{pql}h_{rsl}&=&\sum_{l,p\neq q} \left(-h_{ppl}h_{qql}+h_{pql}^2\right)\\
&=& -|\nabla H|^2+\sum_{l,p, q}h_{pql}^2\geq 0.
\end{eqnarray*}
As $H\geq 1$, put above to (\ref{c}),
\begin{align}\label{e}
0 \geq  \sum_{i} \left(2h_{ii}^2\bar{Ric}(i,i)-2h_{ii}^2\bar{Ric}(\nu,\nu)\right)+\sum_{l,i}\sigma_2^{ii} h_{ll}\bar{R}_{ilil}+\sigma_1^2\sigma_2-CH.
\end{align}

Again by lemma \ref{lemm 3.1},
\begin{align}\label{f}
\sum_{i} \left(h_{ii}^2\bar{Ric}(i,i)-h_{ii}^2\bar{Ric}(\nu,\nu)\right)\geq -h_{11}^2\left(\bar{Ric}(\nu,\nu)-\bar{Ric}(1,1)\right)-C.
\end{align}
Use the fact that $\sigma_2^{ii}=h_{11}+C$ for $i\ge 2$,
\begin{align}\label{g}
\sum_{l,i}\sigma_2^{ii}h_{ll}\bar{R}_{ilil}&=\sum\sigma_2^{ii}h_{11}\bar{R}_{i1i1}+\sum_{l\neq 1}\sigma_2^{ii}h_{ll}\bar{R}_{ilil}\\\nonumber
&\geq h_{11}^2\sum_{i=2}^n\bar{R}_{1i1i}-CH.
\end{align}

Insert (\ref{f}) and (\ref{g})  into (\ref{e}), by Lemma \ref{lemm 3.1},
\begin{eqnarray}\label{pinch}
0\geq  2h_{11}^2\left(\bar{Ric}(1,1)-\bar{Ric}(\nu,\nu)\right)+h_{11}^2\sum_{i=2}^n \bar{R}_{1i1i}+\sigma_1^2\sigma_2-CH.
\end{eqnarray}
Inequality (\ref{pinch}) yields an a priori estimate for $|H|$ in the following cases: (i), $(N^{n+1}, \bar g)$ is Einstein with positive sectional curvature; (ii), if $(N^{n+1}, \bar g)$ is positively curved and its sectional curvature is sufficiently pinched pointwise.

Note that $\sigma_2=\frac{R-\bar R}{2} +\bar Ric(\nu, \nu)$, (\ref{pinch}) becomes
\begin{eqnarray*}
0\geq  2h_{11}^2[\bar{Ric}(1,1)-\bar{Ric}(\nu,\nu)+\frac{1}{2}\sum_{i=2}^n\bar{R}_{1i1i}+\frac{R-\bar R}{4} +\frac{\bar Ric(\nu, \nu)}2]-CH.
\end{eqnarray*}
From this inequality,  one may also impose a direct condition on scalar curvature $R$ of $(M,g)$ to get estimate of $|H|$.

\section{Isometric embeddings}

Existence of isometric embedding of $(\mathbb{S}^2,g)$ in ambient space $(N^3,\bar{g})$ is to solve the following equation system,
\begin{align*}
(d X, d X)_{\bar g}=g,
\end{align*}
where $(\mathbb{S}^2,g)$ is a given manifold with metric $g$, and the inner product is the metric $\bar{g}$.
A standard way to solve the problem is the method of continuity. One connects $(\mathbb{S}^2,g)$ to some $(\mathbb{S}^2,g_0)$
which can be embedded in $(N^3,\bar g)$ through a homotopy path $(\mathbb{S}^2,g_t), 0\le t\le 1$. The equation to solve is
\begin{align}\label{sys1}
(d X_t, d X_t)_{\bar g}=g_t, \forall 0\le t\le 1.
\end{align}
For the openness, one needs to consider the linearized problem:
\begin{align}\label{linearized equation}
(d X_t, D\tau_t)_{\bar g} =q_t, \forall 0\le t\le 1,
\end{align}
for any two symmetric tensor $q_t$, where $\tau_t$ is the variation of $X_t$ and $D$ is the Levi-Civita connection of $(N,\bar{g})$. The openness of the problem has been established in a recent work of Li-Wang \cite{GLW} if system (\ref{sys1}) is elliptic. 

To ensure the ellipticity of system (\ref{sys1}) of the homotopy path, one may impose the following condition 
\begin{eqnarray}\label{e-c}
\min_{x\in \mathbb{S}^2} R(x)> \max_{X\in N, \mu \in ST_{X}N} \{\bar{R}(X)-2\bar{Ric}_{X}(\mu,\mu)\} . 
 .\end{eqnarray} 
We also refer Pogorelov's work \cite{Pb}. For the closedness, Theorem \ref{thm1} and Theorem \ref{thm2} would be suffice if we can establish $C^0$ bound.  If the ambient space if a space form, one can always translate the surface so that the origin is inside $\Sigma$. Therefore $C^0$ is automatic. Also, if the ambient space satisfies condition (2) in Theorem \ref{thm2}, a $C^0$ bound follows from the bound of diameter of the intrinsic metic. This is in particular true for Hadamard space. 
For general ambient space, if $L^2$ norm of solutions $\tau_t$ to equation (\ref{linearized equation}) is under control, then $C^0$ estimate would follow as
\[X(x)=X_0(x)+\int_{0}^{1} \tau_t(x) dt, \forall x\in \mathbb S^2.\]
It is proved in \cite{GLW} that kernel $\mathcal{N}_t$ of
\[(d X_t, D\tau_t)_{\bar g}=0,\]
is of dimension $6$, and
\[\|\tau_t\|_{L^2} \le C, \forall \tau_t \perp \mathcal{N}_t,\]
where $C$ depends on ellipticity constants.
Note that the estimate in Theorem \ref{thm1}  depends on the position of the embedded surface $\Sigma$ in $N$. In a recent work by the second author \cite{Lu}, following Heinz's argument in \cite{H}, a uniform bound for the mean curvature was obtained in the non-degenerate case for embedded surface $(\mathbb S^2, g)$ in $(N^3, \bar g)$. In view of the discussion above, that completes the closedness part of the method of continuity.
In the case that $N$ has no horizon, combining results in \cite{GLW, Lu} and using the normalized Ricci flow as a homotopy path as in \cite{WL}, the existence the isometric embedding of $(\mathbb{S}^2, g)$ to $(N, \bar g)$ can be proved under condition (\ref{e-c}). 

\medskip

More interesting case is that $N$ has a horizon and $\bar g$ is given by
\begin{equation}\label{warp20} \bar{g}=dr^2+\phi^2(r)d\sigma_{\mathbb S^2}^2, \end{equation} where $\phi(r)$ is defined for $r\ge r_0> 0$ and $\phi^{'}(r_0)=0$. Set $s=\phi(r)$, metric $\bar g$ can be rewritten as
\begin{equation}\label{warps0} \bar{g}=\frac{ds^2}{\eta^2(s)}+s^2d\sigma_{\mathbb S^2}^2, \end{equation} where $\eta(s)=\phi^{'}(r(s))$.

Two important examples are the Anti-de Sitter-Schwarzschild manifolds and  the Reissner-Nordstr\"om manifolds in general relativity. A manifold is called static (sub-static), if there exists a function $\eta$ satisfying the static (sub-static) equation
\[(\Delta_{\bar g} \eta) \bar g-\nabla_{\bar g}^2 \eta+\eta Ric_{\bar g} =(\geq)0.\]
Anti-de Sitter Schwarzschild manifolds are static and Reissner-Nordstr\"om manifolds are sub-static. For Anti-de Sitter-Schwarzschild manifold,
\[\eta(s)=\sqrt{1-ms^{1-n}-\kappa s^2}, \quad \overline s\ge s \ge \underline s,\]
where $\kappa\in \mathbb R$, $m>0$, and $\overline s$ and $\underline s$ are the two positive solutions of the equation
\[1-ms^{1-n}-\kappa s^2=0,\]
(if $\kappa\le 0$, $\overline s=+\infty$).
This implies that the Lorentzian warped product
\[\tilde g=-\eta^2dt\otimes dt + \bar g\]
is a solution of Einstein's equations, it is the de Sitter-Schwarzschild space time metric. The Reissner-Nordstr\"om manifold is defined by
$N = \mathbb S^n \times (s_0,\infty)$ and
\[\eta(s)=\sqrt{1- m s^{1-n} + q^2s^{2(1-n)}}.\]
Here, $m > 2q > 0$ are constants, and $s_0$ is defined as the larger of the two
roots of the equation \[1- m s^{1-n} + q^2s^{2(1-n)} = 0.\]

\medskip

Here we want to illustrate that {\it one may not impose condition (\ref{e-c}) indiscriminately in the event of horizon}.
Suppose $(N, \bar g)$ as in (\ref{warp20}), let's consider a fixed slice of $(N,\bar{g})$, $(r,\theta)$. Denote $e_i, i=1,2$ the unit tangent vectors of the slice and $\nu$ the outer normal. The curvature in the ambient space can be computed as
\begin{align*}
\bar{R}_{ijij}=\frac{1-{\phi^\prime}^2}{\phi^2},\quad \bar{R}_{i\nu i\nu}=-\frac{\phi^{\prime\prime}}{\phi}, \quad
\bar{R}=2\frac{1-{\phi^\prime}^2-2\phi\phi^{\prime\prime}}{\phi^2},
\end{align*}
and
\begin{align*}
\bar{Ric}(e_i,e_i)=\frac{1-{\phi^\prime}^2-\phi\phi^{\prime\prime}}{\phi^2},\quad \bar{Ric}(\nu,\nu)=-2\frac{\phi^{\prime\prime}}{\phi}.
\end{align*}
The right hand side of (\ref{e-c}) becomes
\begin{align*}
2\frac{1-{\phi^\prime}^2-2\phi\phi^{\prime\prime}}{\phi^2}-2\min\{2\frac{1-{\phi^\prime}^2}{\phi^2}, -2\frac{\phi^{\prime\prime}}{\phi}\}.
\end{align*}
The ellipticity condition (\ref{e-c}) implies
\begin{align}\label{curv-assump}
R(x)\geq 2\frac{1-{\phi^\prime}^2}{\phi^2}(X(x))\quad and \quad
R(x)\geq -2\frac{1-{\phi^\prime}^2+2\phi\phi^{\prime\prime}}{\phi^2}(X(x))
\end{align}

\medskip

Suppose $(\mathbb{S}^2,g)$ is embedded and it bounds the horizon. Consider the minimum point of the radial function on $(\mathbb{S}^2, g)$, say $x_0$, thus at $x_0$,
$
h^i_j\leq \frac{\phi^\prime}{\phi}\delta^i_j$.
By Gauss equation,
\begin{equation}\label{Ga1}
R(x_0)\leq \frac{2}{\phi^2(r)}, \end{equation}
where $r$ is the radial function at $x_0$.
If (\ref{e-c}) is satisfied, then
\begin{align}\label{curv-assump-2}
R(x)\geq 2\frac{1-{\phi^\prime}^2}{\phi^2}(X(x)), \forall x\in \mathbb S^2.
\end{align}
Since the position function $X$ is not under control in general, that would require
\begin{align}\label{curv-assump-3}
R(x)\geq \frac{2}{\phi^2(r_0)}
\end{align}
where $r_0$ is the radius of the horizon.
Combine (\ref{Ga1}) and  (\ref{curv-assump-3}),
\begin{align*}
\frac{2}{\phi^2(r_0)} \leq R(x_0)\leq \frac{2}{\phi^2(r)}.
\end{align*}
This would yield $\phi(r_0)\geq \phi(r)$.
It contradicts to the fact that $\phi$ is an increasing function. That is, if $(\mathbb S^2, g)\subset (N^3,\bar g)$ contains horizon, ellipticity condition  (\ref{e-c}) can not be valid. This is a main challenge for isometric embedding of surfaces containing horizon. One would like to ask the following question.

\medskip

\noindent
{\it Question: Does there exist a homotopy path with initial embedding containing the horizon such that ellipticity of system (\ref{sys1}) is valid along the path?}

\medskip

On the other hand, one may still isometrically embed $(\mathbb S^2,g)$ to $(N^3,\bar g)$ without containing the horizon. The following is a special example.

\begin{theorem}
Suppose $(N,\bar{g})$ is warped product space with horizon. Suppose $\bar g\in C^4$, suppose there is a convex geodesic ball $B_r\subset N^3$ ($r>0$) and
\begin{eqnarray}\label{curv-assump-newU}
\bar{R}(X)-2\inf\{\bar{Ric}_{X}(\mu,\mu):\mu \in ST_{X}B_r\}\leq C_0, \quad \forall X\in B_r.
\end{eqnarray}
Then for each $\delta>0$, there is $C_{\delta}>0$, such that every $C^4$ metric $g$ on $\mathbb{S}^2$ with scalar curvature $R\ge C_0+\delta$ and $diameter (\mathbb S^2,g)\le C_{\delta}<r$, there is an $C^{3,\alpha}$ isometric embedding $(\mathbb{S}^2,g)$ into $(N,\bar{g})$, $\forall 0<\alpha<1$. In particular, each $C^4$ metric $g$ on $\mathbb{S}^2$ with scalar curvature $R> -2$ can be $C^{3,\alpha}$ isometrically embedded to the Anti-de Sitter-Schwarzschild manifold; each $C^4$ metric $g$ on $\mathbb{S}^2$ with scalar curvature $R> 0$ can be $C^{3,\alpha}$ isometrically embedded to the Reissner-Nordstr\"om manifold.
\end{theorem}

\begin{proof}
The closedness follows from Corollary \ref{cor-3} and discussion above, the openness has been proved in \cite{LW}.  
The only thing left is to find a homotopy path satisfying conditions in theorem. Suppose $p_0$ is the center of the geodesic convex ball $B_r$, the sphere with radius $\epsilon$ centered at $p_0$ has scalar curvature very large if $\epsilon$ is small. Denote this sphere as $B_\epsilon$, we may use $B_\epsilon$ as the initial isometric embedding.

The homotopic path will consist three parts. The first part is the normalized Ricci flow, at the end of this path, we have a metric of constant scalar curvature, i.e. $(\mathbb{S}^2, g_1)$. Clearly, $R_{g_1}\geq C_0+\delta$. If $C_0\geq 0$, then we embed $(\mathbb{S}^2,g_1)$ into Euclidean space as a coordinate sphere. Otherwise, we embed $(\mathbb{S}^2,g_1)$ into hyperbolic space with constant curvature $\frac{C_0}{2}$ as a coordinate sphere. The second part is the normalized Ricci flow for $B_\epsilon$, at the end of this path, we have a metric of constant scalar curvature, i.e. $(\mathbb{S}^2, g_2)$. Clearly the scalar curvature of $(\mathbb{S}^2,g_2)$ is also sufficiently large, thus can be isometrically embedded into Euclidean space or hyperbolic space as a coordinate sphere. The third part is to shrink $(\mathbb{S}^2,g_1)$ in Euclidean space or hyperbolic space to $(\mathbb{S}^2,g_2)$. This can always be done as both are coordinate sphere in ambient space. The path is now finished. By our construction, it's clear that $R\ge C_0+\delta$ along the path. Thus by the same argument, we can establish isometric embedding into $B_r$.

Note that the Anti-de Sitter-Schwarzschild manifold is asymptotic hyperbolic and Reissner-Nordstr\"om manifold is asymptotic flat, for any $r>0$, we may find $B_r$ such that condition (\ref{curv-assump-newU}) is satisfied.
\end{proof}

\medskip

\noindent
{\it Acknowledgement}: The first author would like to thank Mu-Tao Wang for enlightening conversations regarding the isometric embedding problems and the quasi-local masses in general relativity. We would like to thank the anonymous referee for the help in the exposition of the paper.

\end{document}